\documentclass[a4paper,12pt]{amsart}

\usepackage{amsfonts,amssymb,amscd}
\usepackage{amsmath}
\usepackage{graphicx}

\newtheorem{teo}{Theorem}[section]
\newtheorem{prop}[teo]{Proposition}
\newtheorem{lema}[teo]{Lemma}

\newtheorem{defnc}[teo]{Definition}
\newtheorem{coro}[teo]{Corollary}
\newtheorem{example}[teo]{Example}
\newtheorem{conj}[teo]{Conjecture}

\newcommand{\dl}{\lambda}
\newcommand{\da}{\alpha}

\newcommand{\db}{\beta}

\newcommand{\C}{\mathbb{C}}
\newcommand{\Q}{\mathbb{Q}}
\newcommand{\p}{\mathbb{P}}
\newcommand{\D}{\mathbb{D}}

\newcommand{\N}{\mathbb{N}}

\newcommand{\Fcal}{\mathcal{F}}

\newcommand{\mcf}{\mathcal{D}}

\newcommand{\tf}{\widetilde{\mathcal{D}}}
\newcommand{\fol}{\mathcal{F}}
\newcommand{\tilf}{\widetilde{\mathcal{F}}}

\def\rme{{\mathrm e}}

\begin{document}

\title{Topological aspects of completely integrable foliations}

\author{Susana Pinheiro \, \, \, \& \, \, \, Helena Reis}
\thanks{ The first author is partially supported by Funda\c{c}\~ao para a Ci\^encia e
Tecnologia (FCT) through CMUP and through the PhD grant SFRH/BD/61547/2009. The second
author is partially supported by FCT through CMUP and through the project
PTDC/MAT/103319/2008.}

\begin{abstract}
In this paper it is shown that the existence of two independent holomorphic first integrals for foliations by curves
on $(\C^3,0)$ is not a topological invariant. More precisely, we provide an example of two topologically equivalent
foliations such that only one of them admits two independent holomorphic first integrals. The existence of invariant
surfaces over which the induced foliation possesses infinitely many separatrices possibly constitutes the sole
obstruction for the topological invariance of complete integrability and a characterization of foliations admitting
this type of invariant surfaces is also given.
\end{abstract}

\maketitle

\section{Introduction}

A singular holomorphic foliation $\fol$ by curves on a neighborhood of the origin of $\C^n$ is, by definition,
obtained from the local solutions of some holomorphic vector field defined about the origin of $\C^n$ and having
a singular set of codimension at least two. The topology of these singular points
has been widely studied in two and higher dimensional spaces. In the linear context, the topological characterization
of hyperbolic singularities was investigated in \cite{G}, \cite{ladis1}, \cite{ladis2}. These results
were largely extended to the non-linear case in \cite{C-K-P} and, especially, in the work of Chaperon \cite{chaperon}.
Far more general singularities have been systematically studied by Seade, Verjovsky et al. cf. \cite{seade} and references therein.
On the other hand, the study of the topology of integrable systems, or ``nearly" integrable systems,
is a very classical theme, well-represented by the Russian school, for which there is a huge amount of literature.

In the context of singularities of holomorphic foliations in dimension~$2$, these two topics exhibit a remarkable
connection put forward in the seminal paper \cite{M-M}. Indeed, the main result of \cite{M-M} shows that the existence
of a holomorphic first integral for a singular holomorphic foliation defined about $(0,0) \in \C^2$ can be read off
from natural topological conditions. As a consequence, it follows that the existence of a non-constant holomorphic first integral is a
topological invariant of the foliation. In other words, for $n = 2$, consider two local foliations by curves $\fol_1,
\, \fol_2$ that are topologically equivalent in the sense that there is a homeomorphism $h$ defined about $(0,0) \in
\C^2$ and taking the leaves of $\fol_1$ to the leaves of $\fol_2$. Then $\fol_1$ admits a non-constant holomorphic
first integral if and only if so does $\fol_2$ (see~\cite{M-M} and, for a shorter proof, \cite{Mo}).

Possible generalizations of the above mentioned phenomenon have long attracted interest. First, a classical example
attributed to Suzuki and discussed in~\cite{C-M} shows that the existence of a meromorphic first integral is not a
topological invariant. Similarly, for $n = 3$, many experts have wondered whether the existence of two ``independent"
holomorphic first integrals would constitute a topological invariant of the singularity. The paper~\cite{C-S} is an attempt
at initiating a discussion in this direction. However, in the present work, this question will be answered in the negative.
Indeed, we shall prove:

\begin{teo}\label{contra-exemplo}
Denote by $\fol$ and $\mcf$ the foliations associated to the vector fields $X$ and $Y$, respectively, given by
\begin{align*}
X &= 2xy \frac{\partial }{\partial x} + (x^3+2y^2) \frac{\partial }{\partial y} - 2yz \frac{\partial }{\partial z} \, , \\
Y &= x(x - 2y^2 - y) \frac{\partial }{\partial x} + y(x - y^2 - y) \frac{\partial }{\partial y} - z(x - y^2 - y) \frac{\partial }{\partial z} \, .
\end{align*}
The foliations $\fol, \, \mcf$ are topological equivalent. Nonetheless $\fol$ admits two independent holomorphic
first integrals while $\mcf$ does not.
\end{teo}

The example provided by the above mentioned foliations $\fol$ and $\mcf$ is clearly based on Suzuki's foliations
on $(\C^2,0)$ (cf.~\cite{Suzuki1}, \cite{Suzuki2}) since the restrictions of our foliations to the common invariant
plane $\{z = 0\}$ coincide with the latter.

The reader will certainly note that the singular set of the foliations considered in Theorem~\ref{contra-exemplo} is not reduced
to a single point and this might suggest that the ``correct" generalization of Mattei-Moussu theorem involves isolated
singularities. This is actually not at all the case, and to clarify this issue is precisely the aim of the second half of
this paper. As it will be seen, the upshot of Theorem~\ref{contra-exemplo} is that, as the problem was stated, the existence of
two independent holomorphic first integrals may give rise to meromorphic first integrals for the restriction of the
foliation to certain invariant surfaces. This follows from a simple observation that, apparently, was missed in some previous
works, cf. Section~\ref{construction}. It is the presence of these invariant surfaces that ultimately constitutes an essential
obstruction for the topological invariance of ``complete integrability". Throughout this work, a local foliation will be called
completely integrable if it possesses two {\it holomorphic} first integrals that, in addition, are independent in a natural
sense, cf. Definition~\ref{defindependent}. Concerning the role played by the above mentioned invariant surfaces, recall that a
deep study of topological properties of foliations on $(\C^2,0)$ possessing meromorphic first integrals was
conducted by M. Klughertz in~\cite{Martine}. Her techniques yield several examples where ``topological invariance"
for the existence of meromorphic first integrals fails. Relatively simple adaptations of the proof of
Theorem~\ref{contra-exemplo} then enables us to obtain several other examples of foliations on $(\C^3,0)$
for which the ``topological invariance" of the existence of two independent holomorphic first integrals is
not verified. In view of this, and modulo excluding the easier case of foliations possessing non-trivial
linear parts, it is tempting to propose the following:

\begin{conj}\label{conjecture}
Suppose that two foliations by curves on $(\C^3,0)$, $\fol_1, \, \fol_2$, are topologically equivalent and do not
admit invariant surfaces over which the induced foliations are dicritical. Then $\fol_1$ admits two holomorphic
first integrals if and only if so does $\fol_2$.
\end{conj}

Given a foliation $\fol$, throughout this paper a (possibly singular) invariant surface over which the restriction of
$\fol$ defines a dicritical foliation, i.e. a foliation possessing infinitely many separatrices, will be called a
dicritical invariant surfaces.

Let us now come back to the role played by isolated singular points. This has to do with the interaction between
isolated singularities and the existence of dicritical invariant surfaces. Curiously,
modulo very mild generic assumptions, Theorem~\ref{teodicritical} below tells us that these surfaces always
exist provided that the foliation in question has an isolated singularity at the origin. In other words, in view
of the preceding conjecture, the ``correct" generalization of Mattei-Moussu's theorem may involve, in an intrinsic
way, foliations possessing curves of singular points. To make the discussion more accurate, let us now state
Theorem~\ref{teodicritical}.

\begin{teo}\label{teodicritical}
Let $\fol$ be a foliation by curves on $(\C^3,0)$ having an isolated singularity at the origin and admitting
two independent holomorphic first integrals. Suppose that $\tilf$, the transform of $\fol$ by the punctual
blow-up centered at the origin, has only isolated singularities which, in addition, are simple. Then $\fol$
possesses an invariant surface over which the induced foliation is dicritical.
\end{teo}

In the above statement, by a simple singularity, it is meant a singular point of $\fol$ with at least one
eigenvalue different from zero.

Naturally, in the course of the proof of Theorem~\ref{teodicritical}, we shall obtain some additional insight
into the structure of the set of foliations possessing dicritical invariant surfaces. The information collected
there might also be seen as some preliminary steps towards the conjecture stated above.

Along the same ideas, it might be useful to remove the ``generic" condition from the statement of Theorem~\ref{teodicritical}.
Let us then close this Introduction by sketching an argument that might be enough to show that every completely
integrable foliation about the origin of $\C^3$ should possess a dicritical invariant surface. This goes as follows.

Recall that the Seidenberg desingularization theorem for foliations on complex surfaces plays an important role on
the topological characterization of integrable foliations and, in particular, in showing the topological invariance
of integrable foliations. A completely faithful generalization of the Seidenberg result for foliations on $3$-manifolds
cannot exist, since some non-simple singularities are persistent under blow-ups (cf. \cite{P} for details). Nonetheless
final models on a desingularization process of foliations on $3$-manifolds have been described on different papers such
as \cite{C-R-S}, \cite{MQ-P}, \cite{P}. For example, in \cite{P} it is proved that we can obtain only simple singularities
after a finite sequence of ``permissible" blow-ups if weighted blow-ups are allowed. Concerning the standard blow-up Cano
et al. proved the existence of a finite sequence of ``permissible" blow-ups such that the proper transform of $\fol$ only
admits simple singularities with exception of some ``special" singular points of order at most~$2$, \cite{C-R-S}.

The idea to remove the generic condition from Theorem~\ref{teodicritical} is to first show that this type of
singular points cannot appear in the desingularization procedure of $\fol$ provided that $\fol$ is completely
integrable. This goes in the same direction of Lemma~\ref{lemmanonvanishing}~in Section~\ref{dicritical} and
can probably be quickly established by building in the material of \cite{MQ-P}. Once this result is available, we
shall have a foliation $\tilf$ possessing only simple singular points contained in a certain exceptional divisor.
Again, a simple singular point means a point at which the foliation has at least one eigenvalue different
from zero (note that we may have non-isolated singularities). In this situation, when the singularity is
isolated, Lemma~\ref{lemmanonvanishing} still applies to ensure it must have $3$~eigenvalues different from
zero. As to non-isolated singularities, a suitable variable of Lemma~\ref{lemmanonvanishing} should imply
that these singularities have $2$-eigenvalues different from zero. Summarizing, we shall have a foliation
$\tilf$ leaving invariant a more complicated (in particular not irreducible) exceptional divisor but whose
singularities are sufficiently ``well-behaved". Therefore it is likely that a careful study of the possible
arrangements carried out by means of Baum-Bott and Lehmann type index formulas (cf. \cite{lehmann}) may lead
to the desired conclusion.

\section{Construction of a counter example}\label{construction}

Let us begin by giving a formal definition of what is meant by independent holomorphic first integrals.

\begin{defnc}\label{defindependent}
Two non-constant holomorphic first integrals $F, \, G$ for a foliation $\fol$ are said to be independent if
there is no holomorphic function $f$ such that $F = f \circ G$.
\end{defnc}

Let $\fol$ be a foliation on $\left(\C^3,0\right)$ admitting two non-constant and independent holomorphic first
integrals $F$ and $G$. Consider the decomposition of $F$ and $G$ into irreducible factors
\begin{eqnarray*}
F & = & f_1^{m_1}\cdots f_k^{m_k} \\
G & = & g_1^{n_1}\cdots g_l^{n_l} \, \, \, .
\end{eqnarray*}
Suppose that $F$ and $G$ have no common irreducible factors, modulo multiplication by non-vanishing functions.
Then the restriction of $G$ to, for example, $\{f_1=0\}$ is a non-constant holomorphic first integral for the
restriction of $\Fcal$ to the same surface. In particular, the restriction of the foliation $\Fcal$ to $\{f_1
= 0\}$ viewed as a singular foliation defined on a (possibly singular) surface, admits a finitely many
separatrices. In this case, all leaves of $\Fcal|_{\{f_1 = 0\}}$ are ``fully identified" by $G$ in the sense
that the restriction of $G$ to $\{z = 0\}$ provides a non-constant holomorphic first integral for $\Fcal|_{\{f_1
= 0\}}$. Assume now that $f_1$ is a common irreducible factor for $F$ and $G$. Then the restrictions of both
$F$ and $G$ to $\{f_1 = 0\}$ vanish identically. In this case, the leaves of $\Fcal|_{\{f_1 = 0\}}$ cannot be
distinguished by either $F$ or $G$. Nonetheless, it is possible to obtain a non-constant first integral for the
restriction of $\Fcal$ to $\{f_1 = 0\}$ as a function of $F$ and $G$. In fact, the function
\begin{equation}\label{eq1}
\frac{F^{n_1}}{G^{m_1}}=\frac{f_2^{m_2n_1}\cdots f_k^{m_kn_1}}{g_2^{n_2m_1}\cdots g_l^{n_lm_1}}
\end{equation}
is a non-constant first integral of $\Fcal_{|_{\{f_1=0\}}}$. However, in general, this first integral is meromorphic
rather than holomorphic as shown by the simple example below.

\begin{example}
{\rm Consider the holomorphic functions $F=xz$ and $G=xz$ which clearly define two independent holomorphic first integrals
for the foliation associated to the vector field $x \partial /\partial x + y \partial /\partial y - z \partial
/\partial z$. Both $F$, $G$ vanish identically over the invariant manifold $\{x = 0\}$. Nonetheless, the function
$F/G = y/z$ provides a meromorphic first integral for the restriction of $\Fcal$ to this invariant manifold.}
\end{example}

In dimension~$2$, Suzuki provided in \cite{Suzuki2, Suzuki1} an example of a foliation possessing only analytic leaves
but not admitting a meromorphic first integral. Suzuki's motivation was to provide a counter example to a topological
criterion conjectured by Thom for the existence of meromorphic first integrals for foliations by curves on $(\C^2,0)$.
Later Cerveau and Mattei proved that Suzuki's foliation is topologically equivalent to a certain foliation that does
have a non-constant meromorphic first integral. Therefore the existence of meromorphic first integrals is not a
topological invariant contrasting with the case of holomorphic first integrals, cf. \cite{M-M}.

It was, in general, believed that the existence of two independent holomorphic first integrals for foliations on $(\C^3,0)$
should be a topologically invariant characteristic of a foliation. In other words, if two foliations are topologically
equivalent and one of them is completely integrable then so must be the other. Nonetheless, Theorem~\ref{contra-exemplo}
in the Introduction shows that this is not the case. This section is devoted to the proof of Theorem~\ref{contra-exemplo}.
More precisely, we are going to prove that the foliations $\fol, \, \mcf$ associated, respectively, to the vector fields
\begin{align*}
X &= 2xy \frac{\partial }{\partial x} + (x^3+2y^2) \frac{\partial }{\partial y} - 2yz \frac{\partial }{\partial z} \, , \\
Y &= x(x - 2y^2 - y) \frac{\partial }{\partial x} + y(x - y^2 - y) \frac{\partial }{\partial y} - z(x - y^2 - y) \frac{\partial }{\partial z} \, ,
\end{align*}
are topologically equivalent and that $\fol$ is completely integrable while the same does not hold for $\mcf$. The
definition of the foliations in question is itself inspired from the Suzuki and Mattei-Cerveau examples in the following
sense. The plane $\{z = 0\}$ is invariant by both $\fol, \, \mcf$ and, in fact, the restriction of $\fol$ (resp. $\mcf$)
to this invariant manifold coincides with the foliation provided by Cerveau-Mattei (resp. Suzuki). Furthermore they were
chosen so that the projection of each leaf of $\fol$ (resp. $\mcf$) by the projection map ${\rm pr}_2(x,y,z) = (x,y)$ is
still a leaf of $\fol$ (resp. $\mcf$) and, here, we have also added a sort of ``saddle behaviour" for their leaves with
respect to the third component. By ``saddle behaviour" it is meant that as the variable $x$ on the local coordinates of
a leaf decreases to zero, the variable $z$ increases monotonically to exit a fixed neighborhood of the origin.

The fact that the projection through ${\rm pr}_2$ of leaves of $\fol$ is still a leaf of $\fol$ implies that the meromorphic
first integral for the foliation provided by Cerveau-Mattei is also a meromorphic first integral for $\fol$. Let us denote it
by
\[
H_{\fol}(x,y,z) = \frac{y^2 - x^3}{x^2}
\]
However, in view of the observation made in the beginning of this section, this meromorphic first integral for $\fol$ can easily
be split into two independent holomorphic first integrals for $\fol$, namely
\begin{align*}
F_{\fol} (x,y,z) &= (y^2 - x^3)z^2 \, , \\
G_{\fol} (x,y,z) &= xz \, .
\end{align*}
Although the foliation considered by Suzuki does not have any meromorphic first integral, it admits a transcendent first
integral. In turn, this yields a transcendent first integral for the foliation $\mcf$ which will be denoted by
\[
H_{\mcf}(x,y,z) = \frac{x}{y} \rme^{\frac{y(y + 1)}{x}} \ .
\]
By a {\it transcendent first integral} it is meant a first integral that is holomorphic away from a subset of codimension~$1$
and does not admit a meromorphic extension to this subset. The foliation $\mcf$ possesses additional transcendent first
integrals that are independent of the previous one. For example, we can take
\[
G_{\mcf}(x,y,z) = -y \rme^{\frac{y}{x}} z \, .
\]
Naturally the existence of transcendent first integrals does not rule out the possibility of having holomorphic first
integrals as well. Therefore, besides showing that $\fol, \, \mcf$ are topologically equivalent, the proof of
Theorem~\ref{contra-exemplo} also requires the following lemma.

\begin{lema}
$\mcf$ does not admit two independent holomorphic first integrals.
\end{lema}

\begin{proof}
Assume for a contradiction that $\mcf$ admits two independent holomorphic first integrals $F, \, G$. Then both
$F, \, G$ vanish identically over $\{z = 0\}$ for otherwise the restriction of $F$ or $G$ to $\{z = 0\}$ would
provide a non-constant holomorphic first integral for the foliations induced on this invariant plane. This is
clearly impossible since the latter foliation possesses infinitely many separatrices.

Since $F, \, G$ vanish identically over $\{z = 0\}$, it follows that $F = z^k F_1$ and $G = z^l G_1$ for some integers
$k, \, l \geq 1$ and some holomorphic functions $F_1, \, G_1$ (not vanishing identically over $\{z = 0\}$). Since $F$
and $G$ are independent, it follows that $F^l/G^k$ provides a non-constant meromorphic first integral for $\mcf|_{\{z
= 0\}}$. However $\mcf|_{\{z = 0\}}$ coincides with Suzuki's foliation and therefore does not admit any non-constant
meromorphic first integral. The resulting contradiction proves the lemma.
\end{proof}

We are now able to prove Theorem~\ref{contra-exemplo}.

\begin{proof}[Proof of Theorem~\ref{contra-exemplo}]
Given the previous lemma, it remains to prove that $\fol, \, \mcf$ are topologically equivalent. To do this, let
us begin by revisiting the construction of a topological conjugacy between $\fol|_{\{z = 0\}}$ and $\mcf|_{\{z = 0\}}$
as carried out in~\cite{C-M}. In their proof, the $2$-dimensional punctual blow-up of the mentioned foliations at their
singular points was considered. The proper transform of $\fol|_{\{z = 0\}}$ (resp. $\mcf|_{\{z = 0\}}$) by this
blow-up map is a foliation still admitting a meromorphic (resp. transcedent) first integral. The leaves of the new
foliations intersect the exceptional divisor $\left(E \simeq \C\p(1) \right)$ transversely with exception of a single
leaf. This single leaf is tangent to the exceptional divisor at a single point and this point of tangency is of quadratic
type.

Consider the standard affine coordinates $(x,t)$ for the blow-up map $\pi: \widetilde{\C}^2 \rightarrow \C^2$,
where the exceptional divisor is identified with $\{x = 0\}$. The tangency point between $\fol|_{\{z = 0\}}$
(resp. $\mcf|_{\{z = 0\}}$) and the exceptional divisor is given by $t = 0$ (resp. $t = 1$). Let $U_1$ (resp.
$U_0$) be a small neighborhood of the point $(0,1) \in E$ (resp. $(0,0) \in E$) in $\widetilde{\C}^2$. To prove
that $\fol|_{\{z = 0\}}, \, \mcf|_{\{z = 0\}}$ are topologically equivalent, Cerveau and Mattei first constructed
a homeomorphism from $U_1$ into $U_0$ taking the leaves of $\mcf|_{\{z = 0\}}$ to the leaves of $\fol|_{\{z = 0\}}$
on the corresponding open sets. This homeomorphism was subsequently shown to admit an extension (as homeomorphism)
to a neighborhood of the entire exceptional divisor. Note that two foliations are topologically equivalent in a
neighborhood of the origin provided that the correspondent blown-up foliations are topologically equivalent in a
neighborhood of the exceptional divisor.

In our case, we need to construct a topological conjugacy between $\fol, \, \mcf$ on a neighborhood of the origin,
i.e. a homeomorphism defined about the origin of $\C^3$ and taking leaves of $\fol$ to leaves of $\mcf$. For this,
let us first note that the singular sets of $\fol, \, \mcf$ are not reduced to an isolated point. In fact, both sets
coincide with the $z$-axis. Being the $z$-axis invariant for them, it is then natural to consider the blow-up along
the curve of singular points instead of the punctual blow-up at the origin. Let us now proceed to the details.

Consider the standard affine coordinates $(x,t,z)$ (resp. $(u,y,z)$) for the blow-up map centered at the $z$-axis
and given by $\pi_z(x, t, z) = (x, tx, z)$ (resp. $\pi_z(u, y, z) = (uy, y, z)$). We denote by $\widetilde{\C}^3$
the corresponding blow-up of $\C^3$ and by $E$ the resulting exceptional divisor. Note that the pre-image in
$\widetilde{\C}^3$ of a relatively compact neighborhood of the origin is naturally isomorphic to $\widetilde{\C^2}
\times \mathbb{D}$, where $\widetilde{\C^2}$ denotes the blow-up of $\C^2$ by the punctual blow-up at the origin
and $\mathbb{D}$ stands for the unit disc of $\C$. In the affine coordinates $(x,t,z)$ the disc $\mathbb{D}$ is
naturally equipped with the coordinate $z$.

Denote by $\tilf$ (resp. $\tf$) the transform of $\fol$ (resp. $\mcf$) by this blow-up map. The resulting exceptional
divisor $E$ is now invariant by $\tilf, \, \tf$, although both foliations possess leaves intersecting $E$ transversely.
In fact, all leaves contained in the invariant plane $\{z = 0\}$ intersect $E$ transversely with exception of a single
leaf, which is tangent to $E$. The tangency point of $\tilf$ (resp. $\tf$) with $E$ is given, in the coordinates $(x,t,z)$,
by $(0,0,0)$ (resp. $(0,1,0)$).

Let $V_1$ denote a small neighborhood of the point $(0,1,0)$. We begin by presenting a homeomorphism from $V_1$ into
$V_0$, a small neighborhood of $(0,0,0)$, taking the leaves of $\tf$ to the leaves of $\tilf$. It will later be shown
that this homeomorphism admit an extension (as homeomorphism) to a neighborhood of the ``entire" exceptional divisor.

The foliation $\fol$ is completely characterized by the two independent holomorphic first integrals $F_{\fol}, \, G_{\fol}$
or, equivalently, by the two independent meromorphic first integrals $G_{\fol}, \, H_{\fol}$. Thus its transform $\tilf$
must be completely characterized by the two independent functions
\begin{eqnarray*}
\widetilde{G}_{\fol} (x, t, z) & = & G_{\fol} \circ \pi_z (x, t, z) = xz \\
\widetilde{H}_{\fol} (x, t, z) & = & H_{\fol} \circ \pi_z (x, t, z) = t^2 - x  \, .
\end{eqnarray*}
Analogously, the transformed foliation $\tf$ also admits two independent first integrals, namely
\begin{eqnarray*}
\widetilde{G}_{\mcf} (x, t, z) & = & -tx \rme^{t} z \\
\widetilde{H}_{\mcf} (x, t, z) & = & \frac{1}{t} \rme^{t^2x + t}  \ .
\end{eqnarray*}

Following~\cite{C-M}, let $\varphi : (\widetilde{\C}^2, (0,1)) \rightarrow (\widetilde{\C}^2, (0,0))$, $\varphi =
(\varphi_1, \varphi_2)$, be the homeomorphism, in fact diffeomorhism, given by
\begin{eqnarray*}
\varphi_1 (x, t) & = & \widetilde{H}_{\mcf} (0, t) - \widetilde{H}_{\mcf} (x, t) \\
\varphi_2 (x, t) & = & V (t) \ ,
\end{eqnarray*}
where $V(t)$ is a square root of $\widetilde{H}_{\mcf} (0, t) - \widetilde{H}_{\mcf} (0, 1)$, i.e. where $V$ satisfies
$(V(T))^2 = \widetilde{H}_{\mcf} (0, t) - \widetilde{H}_{\mcf} (0, 1)$. Geometrically, this homeomorphism sends straight
lines through the origin into straight lines through the origin, being the correspondence on each line made through
$\varphi_1$ noticing that, for a fixed line, distinct $x$ corresponds to distinct leaves. Now identify $\widetilde{\C}^2$
to $\widetilde{\C}^2 \times \{0\} \subseteq \widetilde{\C}^2 \times \mathbb{D} \simeq \widetilde{\C}^3$. With this
identification, $\varphi$ naturally satisfies $\widetilde{H}_{\fol} \circ \varphi = \widetilde{H}_{\mcf}$. In other
words, the homeomorphism $\varphi$ take leaves of $\tf|_{\{z = 0\}}$ to leaves of $\tilf|_{\{z = 0\}}$. Since
$\widetilde{H}_{\fol}$ and $\widetilde{H}_{\mcf}$ are still independent first integrals for $\tilf, \, \tf$,
respectively, to construct a homeomorphism $\Phi : V_1 \rightarrow V_0$ taking leaves of $\tf$ to leaves of
$\tilf$, $\Phi$ is simply asked to satisfy the following three conditions:
\begin{itemize}
\item[(a)] The restriction of $\Phi$ to $\{z = 0\}$ coincides with $\varphi$
\item[(b)] The first two components of $\Phi$ do not depend on $z$
\item[(c)] $\widetilde{G}_{\fol} \circ \Phi = \widetilde{G}_{\mcf}$
\end{itemize}
where $\Phi = (\Phi_1, \Phi_2, \Phi_3)$. Conditions (a) and (b) imply that $\Phi_1, \, \Phi_2$ coincide with $\varphi_1,
\, \varphi_2$, respectively. Therefore $\widetilde{H}_{\fol} \circ \Phi = \widetilde{H}_{\mcf}$ everywhere.

As already mentioned, $\widetilde{H}_{\fol} \circ \Phi = \widetilde{H}_{\mcf}$ assuming $\Phi_1 = \varphi_1$ and $\Phi_2 =
\varphi_2$. It remains to check that there exists a continuous map $\Phi_3$ such that $\widetilde{G}_{\fol} \circ \Phi =
\widetilde{G}_{\mcf}$ as well. In particular, $\Phi_3$ must admit a continuous extension to the part of the exceptional
divisor contained on $V_1$. Since $\widetilde{G}_{\fol} (x,t,z) = xz$, the product between $\Phi_1$ and $\Phi_3$ must
coincide with $\widetilde{G}_{\mcf}$. In other words, $\Phi_3$ satisfies the equation
\[
(\widetilde{H}_{\mcf}(0, t) - \widetilde{H}_{\mcf}(x, t)) \Phi_3 = - t x \rme^{t}z \ .
\]
Therefore
\begin{equation}\label{phi3_def}
\Phi_3\left(x,t,z\right) = -\frac{t^2x}{1-\rme^{t^2x}}z \ .
\end{equation}
The function $\Phi_3$ is clearly continuous in $V_1 \setminus E$. Furthermore $\Phi_3$ is such that
\[
\lim_{x \rightarrow 0} \phi_3 (x, t, z) = \lim_{x \rightarrow 0} -\frac{t^2 x}{1 - \rme^{t^2 x}} z =
\lim_{x \rightarrow 0} -\frac{t^2}{- t^2 \rme^{t^2 x}} z = z \ ,
\]
by the l'Hopital rule. This means that $\Phi_3$ admits a continuous extension to $V_1 \cap E$. The extension is, in fact,
holomorphic as follows from Riemann extension theorem.

Note that the geometric conditions used in the construction of $\fol$ and $\mcf$ are shared by many other foliations. Our
particular choice of $\fol$ and $\mcf$ was made so as to have the following extra advantage: the above mentioned extension
of $\Phi_3$ to the exceptional divisor $E$ coincides with the identity in $E$.

Let us now show that the homeomorphism $\Phi$ can be extended to a homeomorphism defined in a neighborhood of a compact
part of the exceptional divisor.

Denote by $E_0$ the pre-image of the origin by $\pi_z$, i.e. $E_0 = \pi_z^{-1}(0)$ so that $E_0$ is isomorphic to
$\C\p(1)$. Fix $r > 0$ sufficiently small such that $\Phi$ is continuous in a small neighborhood $W$ of $D_{\mcf,r}$,
where $D_{\mcf,r} \subseteq E_0$ represents the disc of radius $r$ centered at the tangency point $t = 1$. Denote by
$W^{\circ}$ a neighborhood of $S = E_0 \backslash D_{\mcf,r^{\prime}}$, where $0 < r^{\prime} < r$. The restriction
of $\tf$ to $W^{\circ}$ can be viewed as a fibration over $L = E_0 \cap W^{\circ}$
\begin{equation*}
\psi : W^{\circ} \rightarrow L
\end{equation*}
where the projection $\psi$ is defined as follows. Since the leaves of $\tf$ contained in the invariant plane $\{z = 0\}$
are transverse to $E$ on $W^{\circ}$, to every point $a \in W^{\circ} \cap \{z = 0\}$ we associate the unique intersection
point $\psi (a)$ of the leaf through $a$ with $L$. Consider now a point $a \in W^{\circ} \setminus \{z = 0\}$. The leaf
through $a$ does not intersect the exceptional divisor. Nonetheless, it can be projected, through the projection map
${\rm pr}_2$, to a leaf of $\tf|_{\{z = 0\}}$ which is, in turn, transverse to $L$. Thus, assuming that in local
coordinates $a$ is given by $(x_a, t_a, z_a)$, we define the fiber $\psi (a)$ as
\begin{displaymath}
\psi(a) = \left\{ \begin{array}{ll}
t_a  & \textrm{, if $x_a = 0$}\\
\psi(x_a, t_a, 0)  &  \textrm{, otherwise.}
\end{array} \right.
\end{displaymath}
In particular, it follows from the definition of $\psi$ that $(x_a, t_a, z_a)$ and $(x_a, t_a, 0)$ belong to the same
fiber of the fibration in question.

Consider, for simplicity, the change of variable $T = t - 1$ which, basically, ``moves" the tangency point between $\mcf$
and $E$ to the origin. In this coordinate, the disc $D_{\mcf, r}$ is characterized by the condition $|T| < r$. Given $\rho,
\, \varepsilon > 0$, consider the loop in $\{z = 0\}$ defined by
\begin{displaymath}
\gamma : \,  \left\{ \begin{array}{ll}
x(\theta) = \rho \rme^{-i \theta} \\
T(\theta) = \varepsilon \rme^{i \theta} \, .
\end{array} \right.
\end{displaymath}
for $\theta \in [0, 2\pi]$. In coordinates $u = y - x$ and $T^{\prime} = 1/T$, this same loop becomes
\begin{displaymath}
\overline{\gamma} : \,  \left\{ \begin{array}{llll}
u & = & \rho \varepsilon \\
|T^{\prime}|  & = & \frac{1}{\varepsilon} \, .
\end{array} \right.
\end{displaymath}

Fix $\rho, \, \varepsilon$ so that the image of $\gamma$ is contained in $W \cap W^{\circ}$ (for example, $\varepsilon$
can be chosen so that $r^{\prime} < \varepsilon < r$). Then consider on $E_0 (\simeq \C\p(1))$ the image of $\gamma$
through the fibration map $\psi$, $\Gamma = \psi \circ \gamma$. Also denote by $\overline{\Gamma}$ be the image of
$\Gamma$ by the change of coordinates mentioned above ($T = t - 1$). Clearly $\Gamma$ (resp. $\overline{\Gamma}$) is
a loop of index~$1$ around $T = 0$ (resp. $T^{\prime} = 0$). Furthermore, from the expression of $T(\theta)$ on the
definition of $\gamma$, it follows that $\overline{\Gamma}$ corresponds to the boundary of two complementary open
discs in $\C\p(1)$, namely $D_{\varepsilon} =  \{T : |T| < \varepsilon\}$ and $D^{\prime}_{\varepsilon} =  \{T^{\prime}
: |T^{\prime}| < 1/\varepsilon\}$. It is also clear that $D_{\varepsilon} \subseteq W$ while $D^{\prime}_{\varepsilon}
\subseteq W^{\circ}$.

Now consider the compact set
\begin{equation*}
K = \psi^{-1} (\overline{D^{\prime}}_\varepsilon) \cap \{ |y| < \rho \varepsilon \, ,|z| < \varepsilon \} \ .
\end{equation*}
The restriction of $\psi$ to $K$, denoted by $\psi^{\prime}$, is a fibration with base $\overline{D^{\prime}}_\varepsilon$
and whose fibre is isomorphic to $\overline{D}_{\rho \varepsilon} \times \overline{D}_\varepsilon$. Since the base of the
fibration is a disc, and then contractible, we conclude that this fibration is topologically trivial. In particular, the
diagram below commutes.
\[
\renewcommand{\arraystretch}{1.3}
\begin{array}[c]{ccccc}
K & \stackrel{\xi}{\longrightarrow} & \overline{D^{\prime}}_\varepsilon \times \overline{D}_{\rho \varepsilon}
\times \overline{D}_{\varepsilon}\\
& \stackrel{\psi^{\prime}}{\searrow} & \downarrow^{{\rm pr}_1}\\
&  & \overline{D^{\prime}}_\varepsilon
\end{array}
\]
Note that ${\rm pr}_1$ denotes, in local coordinates, the projection to the first component (i.e. ${\rm pr}_1(T,Y,Z) = T$)
and $\xi = (\psi^{\prime}, Y, z)$.

A fibration following the leaves of $\tf$ has just been defined. A similar fibration following now the leaves of
$\tilf$ can also be defined. Let $D_{\fol,r}$ (resp. $D_{\fol,r^{\prime}}$) be the image of $D_{\mcf,r}$ (resp.
$D_{\mcf,r^{\prime}}$) by the conjugating homeomorphism $\Phi$ described above. Those sets are naturally contained
in $E_0$ and they are both isomorphic to discs. Let $V$ be the image of $W$ by $\Phi$. Fix now a small neighborhood
of $E_0 \setminus D_{\fol,r^{\prime}}$ and denote it by $V^{\circ}$. Naturally $V^{\circ}$ can be chosen so that $V^{\circ}
\cap V$ corresponds to $\Phi(W^{\circ} \cap W)$. Then, setting $I = E_0 \cap V^{\circ}$, it follows that $\tilf$ can be
viewed as a fibration over $I$
\[
\psi_1 : W_1^0 \rightarrow I \, .
\]
The fibration $\psi_1$ is defined in perfect analogy with the definition of $\psi$.

Let $\gamma_1$ (resp. $\overline{\gamma}_1$, $\Gamma_1$, $\overline{\Gamma}_1$) be the image of $\gamma$ (resp.
$\overline{\gamma}$, $\Gamma$, $\overline{\Gamma}$) by $\Phi$. Now $\overline{\Gamma}_1$ corresponds to the boundary
of two complementary open sets that are isomorphic to discs. Let us denote by $D_{1,\varepsilon}$ (resp.
$D^{\prime}_{1,\varepsilon}$) the ``disc" containing $t=0$ (resp. $t = \infty$).

Finally, let $K_1$ be a compact neighborhood of $D^{\prime}_{1,\varepsilon}$ so that
\[
K_1 \cap \psi_1^{-1} (\overline{\Gamma}_1) = \overline{U}_1 \cap \psi_1^{-1} (\overline{\Gamma}_1) \, ,
\]
where $U_1 = \Phi(U)$. Furthermore $K_1$ is chosen as a compact neighborhood of $D^{\prime}_{1,\varepsilon}$ such that
the restriction of $\psi_1$ to $K_1$ is still a fibration with fiber isomorphic to $\overline{D}_{\rho \varepsilon}
\times \overline{D}_{\varepsilon}$. This restriction is again a topologically trivial fibration since the base is
isomorphic to a disc.

The homeomorphism $\Phi$ induces a bundle morphism $\Phi^0$ between $(\psi^{\prime})^{-1}(\overline{\Gamma})$ and
$(\psi_1^{\prime})^{-1}(\overline{\Gamma}_1)$. To conclude that $\tilf$ and $\tf$ are topologically equivalent, it
suffices to show that this homeomorphisms admits a fibration-preserving continuous extension.

Both $\overline{D}^{\prime}_\varepsilon$ and $\overline{D}^{\prime}_{1,\varepsilon}$ are homeomorphic to the unit
disc $\D$. Denote by $\lambda$ (resp. $\lambda_1$) a homeomorphism from $\overline{D}^{\prime}_\varepsilon$ (resp.
$\overline{D}^{\prime}_{1,\varepsilon}$) onto $\D$. The homeomorphism $\Phi^0$ induces a homeomorphism $\overline{v}$
between the boundary of the unit discs, $\overline{v} : \partial \D \rightarrow \partial \D$. In turn, $\overline{v}$
can naturally be extended to $\D$ over the radial lines, i.e. by sending the line of angle $\theta$ onto the line of
angle $\overline{v}(\theta)$. In polar coordinates, this extension is given by
\begin{equation*}
\overline{V}(r, \theta) = (r, \overline{v}(\theta)) \, .
\end{equation*}

Let us now consider $V = \dl_1^{-1} \circ \overline{V} \circ \dl$. It represents a continuous extension of $\Phi^0$
to $\overline{D}_{\varepsilon}^{\prime}$ and this extension yields the following commutative diagram
\[
\renewcommand{\arraystretch}{1.3}
\begin{array}[c]{ccccc}
\overline{D}^{\prime}_{\varepsilon} & \stackrel{V}{\longrightarrow} & \overline{D}^{\prime}_{1,\varepsilon} \\
\downarrow^{\lambda} & & \downarrow^{\lambda_1} \\
\D & \stackrel{\overline{V}}{\longrightarrow} & \D
\end{array}
\]
In other words, it has been established a continuous correspondence between the bases of the trivial fibrations $\psi$
and $\psi_1$. The next step is to set up a correspondence between the fibers of the fibrations in question. The bundle
morphism $\Phi^0$ gives us, already, a correspondence between the fibers over the boundaries of the base, i.e. over
$\partial D^{\prime}_{\varepsilon}$ and $\partial D^{\prime}_{1,\varepsilon}$, leading to the commutative diagram
below
\[
\renewcommand{\arraystretch}{1.3}
\begin{array}[c]{ccccc}
\partial \overline{D}^{\prime}_{\varepsilon} \times \overline{D}_{\rho \varepsilon} \times \overline{D}_\varepsilon
& \stackrel{\Phi^0}{\longrightarrow}
& \partial \overline{D}^{\prime}_{1,\varepsilon} \times \overline{D}_{\rho \varepsilon} \times \overline{D}_\varepsilon \\
\downarrow & & \downarrow \\
\partial \D & \stackrel{\overline{v}}{\longrightarrow} & \partial \D
\end{array}
\]
Let ${\rm Diff}_0 (\overline{D}_{\rho \varepsilon} \times \overline{D}_\varepsilon,0)$ denote the group of homeomorphisms of
$\overline{D}_{\rho \varepsilon} \times \overline{D}_\varepsilon$ preserving the origin. Consider also the parametrization
of $\partial \overline{D}^{\prime}_{\varepsilon}$ by the polar coordinate $\theta$. Since the fibers are all isomorphic to
$\overline{D}_{\rho \varepsilon} \times \overline{D}_\varepsilon$, the bundle morphism $\Phi^0$ induces a continuous
$1$-parameter family $\{h_{\theta}\}_{\theta \in [0, 2\pi]}$ of elements of ${\rm Diff}_0 (\overline{D}_{\rho \varepsilon}
\times \overline{D}_\varepsilon,0)$.

\begin{lema}\label{homotopiatrivial}
The homotopy class of $\{h_{\theta}\}_{\theta \in [0, 2\pi]}$ in ${\rm Diff}_0 (\overline{D}_{\rho \varepsilon} \times
\overline{D}_\varepsilon,0)$ is trivial.
\end{lema}

Before proving this lemma, note that $h = \Phi^0$ admits a homeomorphic extension to $\overline{D}^{\prime}_{\varepsilon}
\times \overline{D}_{\rho \varepsilon} \times \overline{D}_{\varepsilon}$. To check this, assume that the homotopy class
of $\{h_{\theta}\}_{\theta \in [0, 2\pi]}$ is trivial. Fix $\theta \in [0,2\pi]$ and let
\[
\Xi_{\theta} : \overline{D}_{\rho \varepsilon} \times \overline{D}_{\varepsilon} \times [0,1] \rightarrow
\overline{D}_{1,\rho \varepsilon} \times \overline{D}_{1,\varepsilon}
\]
be the homotopy map joining $h_0$ to $h_{\theta}$, i.e. let $\Xi_{\theta}$ be the map such that $\Xi(u,z,0) = h_0(u,z)$,
$\Xi(u,z,1) = h_{\theta}(u,z)$ and $\Xi(0,0,s) = (0,0)$ for all $s \in [0,1]$. Then the map defined by
\[
H \left((r,\theta),u,z \right) = \left( \overline{V}(r, \theta), \Xi_{\theta}(u,z,r) \right) \, ,
\]
is a continuous extension of $h$ to $\overline{D}^{\prime}_{\varepsilon} \times \overline{D}_{\rho \varepsilon} \times
\overline{D}_{\varepsilon}$. This homeomorphism induces a natural homeomorphism from $K$ onto $K_1$ and this ends the
proof of the theorem.
\end{proof}

It remains however to prove Lemma~\ref{homotopiatrivial}.

\begin{proof}[Proof of Lemma~\ref{homotopiatrivial}]
The elements $h_{\theta}$, $\theta \in [0, 2\pi]$, correspond to the restriction of $h$ to the fiber over $\theta \in \partial
\overline{D}^{\prime}_{\varepsilon}$. More precisely,
\[
h_{\theta} = h|_{(\theta \times \overline{D}_{\rho \varepsilon} \times \overline{D}_\varepsilon)} \, .
\]
In particular, $h_{\theta} (0) = 0$ for all $\theta \in [0, 2\pi]$. Since the projection of leaves of $\fol, \, \mcf$
through the projection map ${\rm pr}_2$ are still leaves of $\fol, \, \mcf$, respectively, it follows that $h_{\theta}$
has the form
\[
h_{\theta}(u,z) = (h_{1,\theta}(u), h_{2,\theta}(u,z)) \, ,
\]
where $h_{1,\theta}$ (resp. $h_{2,\theta}$) represents the restriction of $\Phi_1$ (resp. $\Phi_3$) to a fixed $\theta
\in \partial \overline{D}^{\prime}_{\varepsilon}$. In particular, $\{h_{1,\theta}\}$ coincides with the correspondent
$1$-parameter family of homeomorphism constructed in \cite{C-M}. Furthermore, this $1$-parameter family has trivial
homotopy class in ${\rm Diff}_0 (\overline{D}_{\rho \varepsilon},0)$ (cf. \cite{C-M}). Let $\aleph_{\theta}$ denote
a homotopy map joining $h_{1,0}$ to $h_{1,\theta}$, i.e. let $\aleph_{\theta}$ be a map such that $\aleph(u,0) =
h_{1,0}(u)$, $\aleph(u,1) = h_{1,\theta}(u)$ and $\aleph(0,s) = 0$ for all $s \in [0,1]$. Then a homotopy map
$\Xi_{\theta}$ joining $h_0$ to $h_{\theta}$ can easily be constructed. In fact, recalling the expression of $\Phi_3$,
$\Xi_{\theta}$ can be given by
\[
\Xi_{\theta}(u,z,s) = \left( \aleph_{\theta}(u,s) , -\frac{(1 - e^{is\theta})^2 \aleph_{\theta}(u,s)}{1 - e^{(1 -
e^{is\theta})^2 \aleph_{\theta}(u,s)}} \, z \right)
\]
\end{proof}

\section{On dicritical invariant surfaces}\label{dicritical}

Note that both foliations $\fol$ and $\mcf$ considered in the preceding possess an invariant surface over which the
induced foliation is dicritical, i.e. they share a common dicritical invariant surface. Furthermore, the singular
sets of $\fol$ and $\mcf$ are not reduced to isolated singular points. Examples of vector fields admitting two
independent holomorphic first integrals and such that their associated foliations possess either
\begin{itemize}
\item[(a)] an isolated singular point and dicritical invariant surfaces or
\item[(b)] a curve of singular points but no dicritical invariant surfaces
\end{itemize}
can easily be constructed. For example, the vector field $X = x \partial /\partial x + y \partial /\partial y -
z \partial /\partial z$ induces a foliation of type (a). The origin is the unique singular point and the foliation
induced on the invariant plane $\{z = 0\}$ is dicritical. Moreover it admits two independent holomorphic first
integrals, namely $F(x,y,z) = xz$ and $G(x,y,z) = yz$. In turn, the vector field $Y = x \partial /\partial x +
y \partial /\partial y$, which admits $F(x,y,z) = xy$ and $G(x,y,z) = z$ as holomorphic first integrals, induces
a foliation of type (b). In this section, it will be proved that under ``generic" conditions there is no completely
integrable foliation with an isolated singular point and without dicritical invariant surfaces. This is the contents
of Theorem~\ref{teodicritical}.

Before proving Theorem~\ref{teodicritical}, some comments should be made. First of all, it should be noted that the
existence of dicritical invariant surfaces depends on the existence of non-trivial common factors between the
decomposition into irreducible factors of two independent holomorphic first integrals. In fact, let $F$ and $G$
be two independent holomorphic first integrals for a foliation $\fol$ and let $f_i, \, g_j$ be their irreducible
factors. The leaves accumulating at the origin and, in particular, the separatrices of $\fol$, are all contained in
the union of the $2$-dimensional invariant varieties given by $\{f_i = 0\}$ and by $\{g_j = 0\}$. Assume that $F, \,
G$ admit only trivial common factors, modulo multiplication by a non-vanishing function. Then the restriction of $F$
(resp. $G$) to each irredubible component $\{g_j = 0\}$ (resp. $\{f_i = 0\}$) provides a non-constant holomorphic
first integral for the restriction of $\fol$ to the same surfaces. Therefore, the restriction of the foliation $\fol$
to all invariant surfaces through the origin admits a finite number of separatrices.

A complete characterization, in terms of common irreducible factors, of the foliations admitting dicritical invariant
surfaces can be obtained. More precisely, let
\begin{eqnarray*}
F & = & h_1^{k_1}\cdots h_p^{k_p} f_1^{\da_1} \cdots f_q^{\da_q} \\
G & = & h_1^{l_1}\cdots h_p^{l_p} g_1^{\db_1} \cdots g_r^{\db_r} \, \, \, .
\end{eqnarray*}
be the decomposition of two independent holomorphic first integrals $F, \, G$ into irreducible factors, where
$h_1, \ldots, h_p$ represent the common factors. Naturally it is not excluded the possibility of having $h_i$
constant equal to one for all $i=1, \ldots, p$. Similarly, it may occur that all non-trivial factors of $F, \,
G$ are, indeed, common (i.e. $f_i$ and $g_j$ are all constant equal to one). To abridge notations, these cases
will be referred to by saying that $p=0$ or that $q=0$, $r=0$, respectively. Note that, when $q = r = 0$, we
necessarily have $p \geq 2$, otherwise $F$ and $G$ would be dependent.

Without loss of generality, we can assume that $h_1, \ldots, h_p$ are ordered so that
\begin{equation}\label{order}
\frac{k_1}{l_1} \leq \cdots \leq \frac{k_p}{l_p} \, .
\end{equation}
Furthermore, if all the inequalities above are, in fact, equalities, then both $q, \, r$ should be assumed greater
than or equal to~$1$. Indeed, at least one between $q, \, r$ must be greater than~$1$, otherwise $F, \, G$ would be
dependent. The case that only one between $q, \, r$ is strictly positive can easily be transformed into the case
where $F, \, G$ do not admit common irreducible factors. From now on, we shall assume that $F, \, G$ are independent
first integrals in their simplified form in the sense above. Under this assumption:

\begin{prop}
The foliation $\fol$ possesses a dicritical invariant surface if and only if one of the following cases occurs:
\begin{itemize}
\item[(1)] in~(\ref{order}) there are at least two distinct strict inequalities
\item[(2)] in~(\ref{order}) there are exactly one strict inequality and at least one between $q, \, r$ is greater than
or equal to~$1$
\item[(3)] in~(\ref{order}) there is no strict inequalities and $p \geq 1$ (note that $q, \, r$ are both assumed greater
than or equal to~$1$ in this case).
\end{itemize}
\end{prop}

\begin{proof}
It was already checked that, if $F, \, G$ do not admit irreducible common factors, then $\fol$ does not admit
dicritical invariant surfaces. So, let us assume that $F, \, G$ possess at least one non-trivial irreducible
common factor.

Assume first that in~(\ref{order}) there are two or more distinct strict inequalities. Fix $1 < i < p$ such that
\[
\frac{k_1}{l_1} < \frac{k_i}{l_i} < \frac{k_p}{l_p} \, .
\]
Then $F^{l_i}/G^{k_i}$ is necessarily a meromorphic (non holomorphic) first integral for the restriction of $\fol$ to
the invariant surface $\{f_i = 0\}$. Indeed, the estimate above implies that the power of $h_1$ in $F^{l_i}/G^{k_i}$
is strictly negative while the power of $h_p$ is strictly positive.

Assume now that~(\ref{order}) possesses exactly one strict inequality. In this case, it follows that $F, \, G$ can
be written in the form $F = a_1^{k_1} a_2^{k_2} g$ and $G = a_1^{l_1} a_2^{l_2} h$ where $a_1, \, a_2, \, g, \, h$
are not necessarily irreducible, but do not admit any common irreducible factors among them. Furthermore, we have
$k_1/l_1 < k_2/l_2$. Suppose first that at least one between $q, \, r$ is greater than or equal to~$1$. Assuming
$q \geq 1$, i.e. assuming that $g$ is a non-constant function vanishing at the origin, we conclude that
\[
\frac{F^{l_2}}{G^{k_2}} = a_1^{k_1l_2 - l_1k_2} \frac{g^{l_2}}{h^{k_2}}
\]
is a meromorphic (non-holomorphic) first integral for the foliation induced on $\{a_2 = 0\}$ since the power of $a_1$
is strictly negative and $g$ is not invertible. The induced foliation on $\{a_2 = 0\}$ is therefore dicritical. Suppose
now that $q = r = 0$. Then $F, \, G$ are simply given by $F = a_1^{k_1} a_2^{k_2}$ and $G = a_1^{l_1} a_2^{l_2}$. The
only invariant surfaces over which the induced foliation can be dicritical are those given by the equations $a_1 = 0$
and $a_2 = 0$ and, in any event, it follows that both $F, \, G$ vanish identically over the invariant surfaces in question.
In this case, however, it follows that $a_i$ constitutes a non-constant holomorphic first integral for the induced foliation
over $\{a_j = 0\}$, for $i \ne j$. Thus the induced foliations cannot be dicrictical over $\{a_j = 0\}$.

Finally, it remains to consider the case where~(\ref{order}) does not admit strict inequalities. As already
mentioned, in this case $q, \, r$ are both assumed greater than or equal to~$1$. Therefore $F = a^k g$ and
$G = a^l h$ for some non-constant holomorphic functions $a, \, g, \, h$ vanishing at the origin and without
non-trivial common factors among them. Therefore $g^l/h^k$ is a meromorphic (non-holomorphic) first integral
over the invariant variety $\{a = 0\}$. The result follows.
\end{proof}

From now on, let us assume that $\fol$ admits two independent holomorphic first integrals and possess an isolated
singular point at the origin. Under these assumptions, we first note the following

\begin{lema}
$\fol$ admits a separatrix.
\end{lema}

\begin{proof}
Let $F$ and $G$ be two independent holomorphic first integrals for the holomorphic foliation $\fol$ and let us
consider their decomposition into irreducible factors
\begin{align*}
F &= f_1^{\da_1} \cdots f_k^{\da_k} \\
G &= g_1^{\db_1} \cdots g_l^{\db^l} \, ,
\end{align*}
where $f_1, \, \ldots , \, f_k, \, g_1, \, \ldots , \,  g_l$ are holomorphic functions and $n_1, \, \ldots , \,
n_k , \, m_1, \, \ldots , \, m_l \in \N$. If $\fol$ admits a separatrix then the separatrix must be contained in
the intersection of the irreducible component $\{f_i = 0\}$ with an irreducible component $\{g_j = 0\}$, for some
$i, j$. Consider the irreducible component $\{f_1 = 0\}$. If there exists an irreducible component $\{g_j = 0\}$
do not coinciding with $\{f_1 = 0\}$, then the intersection $\{f_1 = 0\} \cap \{g_j = 0\}$ corresponds to a
separatrix of $\fol$. In fact, the two irreducible components are invariant by $\fol$ and thus so is their intersection,
unless it is contained in the singular set of $\fol$. The latter case does not occur since the origin is supposed
to be an isolated singular point of $\fol$.

To conclude the proof we claim that $F, \, G$ possess at least two distinct irreducible factors. Indeed, assume
for a contradiction that both $F, \, G$ possess exactly one irreducible factor and that this factor is common
for $F, \, G$. This means that $F = f^p$ and $G = f^q$ for some $p, \, q \in \N$ and an irreducible factor $f$.
This implies that $F, \, G$ are not independent, contradicting our assumption.
\end{proof}

Summarizing the preceding proof, we have seen that the intersection of two distinct invariant surfaces for $\fol$
defines a curve that either is a separatrix of $\fol$ or it is contained in the singular set of $\fol$. Naturally
the second possibility cannot occur if singularities are supposed to be isolated. Next, we have:

\begin{lema}\label{lemmanumberseparatrices}
Let $\fol$ be as above and assume that it does not admit dicritical invariant surfaces. Then $\fol$ possesses a
finite number of separatrices. Moreover, the separatrices are the unique leaves accumulating at the singular
point.
\end{lema}

\begin{proof}
The leaves of $\fol$ accumulating at the origin must be contained in the invariant varieties $\{f_i = 0\}$ or
$\{g_j = 0\}$ for some $i, j$. Since the restriction of $\fol$ to the invariant surfaces is supposed to be
non-dicritical, it follows that a suitable combination of $F$ and $G$ leads us to a non-constant holomorphic
first integral over each one of the invariant varieties above. In particular, the leaves accumulating at the
origin must coincide with the separatrices. Furthermore, the separatrices must be contained in the intersection
of an irreducible component $\{f_i = 0\}$ with another one of the form $\{g_j = 0\}$. Clearly there is a finite
number of possible combinations and the result follows.
\end{proof}

Let $\tilf$ represent the proper transform of $\fol$ by the punctual blow-up at the origin. Denote by $E$ the
resulting component of the exceptional divisor and let $\tilf_E$ denote the foliation induced by $\tilf$ on $E$.
To prove Theorem~\ref{teodicritical} we must assume that $\tilf$ only admits isolated singular points and that
these singular points are simple. The proof is divided in different steps beginning with the following.

\begin{prop}\label{propmerofirstintegral}
The foliation induced by $\tilf$ on the exceptional divisor admits a non-constant meromorphic first integral.
\end{prop}

\begin{proof}
Let $F, \, G$ be two independent holomorphic first integrals for $\fol$ and consider the pull-back of $F, \, G$ by
the punctual blow-up map. In standard affine coordinates $(u,v,z)$ for the blow-up map $\pi : \widetilde{\C}^3
\rightarrow \C^3$, where the exceptional divisor is identified with $\{z = 0\}$, the pull-backs of $F, \, G$ are
given by
\begin{align*}
\widetilde{F} (u,v,z) &= z^k \overline{F} (u,v,z) \\
\widetilde{G} (u,v,z) &= z^l \overline{G} (u,v,z) \,
\end{align*}
respectively, where $\overline{F}$ and $\overline{G}$ are holomorphic functions not divisible by $z$ and where $k,
\, l \in \N^{\ast}$. Since $F$ and $G$ are independent, it follows that $\widetilde{F}^l /\widetilde{G}^k$ defines
a non-constant meromorphic first integral for the blown-up foliation.

Note that the preceding statement does not imply that the restriction of $\widetilde{F}^l/\widetilde{G}^k$ to $\{z = 0\}$
is not constant as well. However, when this restriction is not constant, then it induces a meromorphic first integral for
the induced foliation on $E \simeq \C\p(2)$ and the proposition results at once. So, let us assume that $\widetilde{F}^l
/\widetilde{G}^k$ is constant over $\{z = 0\}$. This is equivalent to saying that the first non-zero homogeneous components
of $F^l$ and $G^k$ coincide up to a non-vanishing constant $\alpha$, i.e. $G^k = \alpha F^l$ for some $\alpha \in \C^{\ast}$.
In this case, the process above should be repeated with the first integrals $F$ and $H$, with $H = F^l - \alpha G^k$. Note
that $H$ still is a non-constant holomorphic first integral for $\fol$ independent of $F$.

Denote by $p$ the order of $H$. We have that $p$ is greater than $kl$ since the first non-trivial homogeneous components
of $F^l$ and of $\alpha G^k$ cancel each other on $H$. Naturally the fact that these components cancel each other out, does
not ensure that the first non-zero homogeneous component of $F^p$ and $H^k$ must be distinct up to a multiplicative factor.
In other words, it may still happen that $\widetilde{F}^p/\widetilde{H}^k$ is constant over $\{z = 0\}$ and hence it does
not yield a meromorphic first integral for the foliation on $\C\p(2)$. However, when the restriction of $\widetilde{F}^p
/\widetilde{H}^k$ to $\{z = 0\}$ is constant, we repeat the process once again for $F$ and $I = F^p - \beta H^k$ where
$\beta$ plays the same role for $F, \, H$ as $\alpha$ for $F, G$.

\smallbreak

{\bf Claim:} This process stops after finitely many steps.

\smallbreak

To prove the claim it suffices to show that, if this process does not stop, then $F, \, G$ are, in fact, dependent. So, let
us consider $F, \, G$ as above and suppose that the first non-zero homogeneous components of $F, \, G$ are powers of a same
homogeneous polynomial. Indeed, modulo replacing $F$ (resp. $G$) by a suitable power of it, we can assume without loss of
generality that these first non-zero components differ by a multiplicative constant. It will be proved that, if this process
does not stop then
\[
G = \sum_{i=1}^{\infty} \alpha_i F^i \,
\]
where $F^i$ denotes the $i$-th power of $F$ and $\alpha_i$ is the multiplicative constant on each step. Induction will be
used to prove this result. More precisely, it will be proved the following.

Let $F = F_k + F_{k+1} + F_{k+2} + \ldots$ (resp. $G = G_k + G_{k+1} + G_{k+2} + \ldots$) be the decomposition of $F$
(resp. $G$) in homogeneous components.
\begin{itemize}
\item[(1)] If $G_k = \alpha_1 F_k$ and the process does not stop, then
\[
G = \alpha_1 F + G^{(1)}
\]
where the first non-trivial homogeneous component of $G^{(1)}$ has order at least $2k$ and $G^{(1)}_{2k} = \alpha_2 F^2_k$.

\item[(2)] Suppose that $G = \alpha_1 F + \cdots + \alpha_iF^i + G^{(i)}$ with $G^{(i)}$ of order at least $(i+1)k$ and
such that $G^{(i)}_{(i+1)k} = \alpha_{i+1} F_k^{(i+1)}$. If the process does not stop, then
\[
G = \alpha_1 F + \cdots + \alpha_{i+1} F^{i+1} + G^{(i+1)}
\]
where the first non-trivial homogeneous component of $G^{(i+1)}$ has order at least $(i+2)k$ and $G^{(i+1)}_{(i+2)k} =
\alpha_{i+2} F_k^{i+2}$.
\end{itemize}
If both conditions are proved, then the result follows at once. Let us prove them.

\begin{itemize}
\item[(1)] Suppose that $G_k = \alpha_1 F_k$ and let $G^{(1)} = G - \alpha_1 F$. Denote by $p$ the order of $G^{(1)}$, being
$p > k$. The foliation induced by $\fol$ on $E$ cannot be defined by the restriction of $\widetilde{F}^p /(\widetilde{G}^{(1)})^k$
to $\{z = 0\}$ if and only if the first non-zero homogeneous components of $\widetilde{F}^p$ and $(\widetilde{G}^{(1)})^k$
coincide up to a multiplicative constant $\alpha_2$. This means that the first non-trivial homogenous component of $\widetilde{G}^{(1)}$
is a power of $F_k$. In particular, we obtain that $G_j = \alpha_j F_j$ for all $k \leq j < 2k$. Thus $G$ can be written as
\[
G = \alpha_1 F + G^{(1)}
\]
with $G^{(1)}$ as described above. Note that the multiplicative constant $\alpha_2$ might be equal to zero.

\item[(2)]  Suppose now that $G = \alpha_1 F + \cdots + \alpha_iF^i + G^{(i)}$ with $G^{(i)}$ as above. Let
\[
G^{(i+1)} = G^{(i)} - \alpha_{i+1} F^{i+1} \, .
\]
The order $p$ of $G^{(i+1)}$ is greater than $(i+1)k$. Again the foliation $\fol_E$ cannot be defined by the restriction of
$\widetilde{F}^p /(\widetilde{G}^{(i+1)})^k$ to $\{z = 0\}$ if and only if the first non-zero homogeneous component of
$G^{(i+1)}$ is a power of $F_k$. Therefore $G^{(i+1)}$ has order at least $(i+2)k$ and $G^{(i+1)}_{(i+2)k} = \alpha_{i+2}
F_k^{i+2}$ for some $\alpha_{i+2} \in \C$. Furthermore
\[
G_j^{(i)} = \alpha_{i+1} F_j^{i+1}
\]
for all $(i+1)k \leq j < (i+2)k$. The result follows.
\end{itemize}
\end{proof}

As an immediate consequence of the previous lemma, we have the following.

\begin{coro}\label{corolary}
Let $p \in E$ be a singular point for $\tilf$ and let $\lambda, \mu$ denote the eigenvalues of $\fol_E$ at $p$.
Then $\lambda \ne 0$ if an only if $\mu \ne 0$. Furthermore, if $\lambda, \mu$ are distinct from zero then
$\lambda /\mu \in \Q$.
\end{coro}

Next we have:

\begin{lema}\label{lemmanonvanishing}
Let $p \in E$ be a simple singular point for $\tilf$. Then none of the eigenvalues of $\tilf$ at $p$ is equal to zero.
\end{lema}

\begin{proof}
Let $\lambda_1, \lambda_2, \lambda_3$ denote the three eigenvalues of $\tilf$ at $p$ and assume that $\lambda_3$ corresponds
to the eigenvalue associated to the direction transverse to $E$. By assumption, not all the eigenvalues $\lambda_1, \lambda_2,
\lambda_3$ are equal to zero. To prove the lemma, let us assume for a contradiction that at least one of them
is equal to zero.

By using Corollary~\ref{corolary}, suppose first that $\lambda_1 = \lambda_2 = 0$. Thus $\lambda_3 \ne 0$ and it follows that
$\tilf$ possesses a separatrix $S$ transverse to $E$. Since $\fol$ admits two independent holomorphic first integrals, we conclude
that the separatrix $S$ must be contained in an invariant surface $M$ for $\tilf$ which is, in addition, transverse to $E$.
Naturally the intersection of $M$ with the exceptional divisor is a (local) analytic curve which is invariant by the foliation
induced on $E$ and this curve clearly contains the singular point in question. Besides, the mentioned curve is not contained in
the singular set of either foliations since these have isolated singularities. It then follows that the intersection of $M$
and $E$ defines a separatrix for the induced foliation on $E$. Consider now the restriction $\tilf_M$ of $\tilf$ to $M$.
The point $p$ belongs to the singular set of $\tilf_M$. More precisely, $p$ is a singular point of saddle-node type for
$\tilf_M$ what immediately yields a contradiction since $\fol$ is completely integrable so that $\tilf_M$ possesses a
non-constant first integral.

Assume now that $\lambda_3 = 0$. Then $\lambda_1 . \lambda_2 \ne 0$. The standard Poincar\'e-Dulac Theorem guarantees that $\tilf$
admits a formal separatrix $\hat{S}$ that is, in addition, transverse to $E$. Modulo performing finitely many blow-ups,
this formal separatrix admits
a (formal) parametrization through the variable $z$. In view of Ramis-Sibuya theorem, \cite{R-S}, there must exist an actual
leaf $S$ of $\tilf$ accumulating at $p$ and admitting an analytic parametrization defined on a certain open sector $V$ with coordinate
$z$. Furthermore, it admits $\hat{S}$ as its asymptotic expansion. Since $\tilf$ is completely integrable, it
follows that $S$ is an analytic separatrix for $\tilf$ at $p$. Moreover, this separatrix is transverse to $E$ since
it is asymptotic to $\hat{S}$ which, in turn, is (formally) transverse to $E$. The proof now follows as in the previous case.
\end{proof}

We are now able to prove Theorem~\ref{teodicritical}.

\begin{proof}[Proof of Theorem~\ref{teodicritical}]
Let $\fol$ be a foliation as stated in Theorem~\ref{teodicritical} and assume for a contradiction that $\fol$ does not
admit an invariant surface over which the induced foliation is dicritical.

Fix a singular point $p \in E$ and denote by $\lambda_1, \, \lambda_2, \, \lambda_3$ the eigenvalues of $\tilf$ at $p$.
From Lemma~\ref{lemmanonvanishing}, we know that none of these eigenvalues is equal to zero. Fix a separatrix through $p$
and consider the holonomy map with respect to this separatrix. The eigenvalues of the linear part of the holonomy map are
given, up to a relabeling of the eigenvalues, by $e^{2\pi i \lambda_2/ \lambda_1}$ and $e^{2\pi i \lambda_3/ \lambda_1}$.
Since $\fol$ is completely integrable, the holonomy map is periodic which, in turn, implies that $\lambda_i /\lambda_j
\in \Q^{\ast}$ for all $i \ne j$.

\smallbreak

{\bf Claim:} Let $\lambda_3$ denote the eigenvalue associated to the direction transverse to the component of the exceptional
divisor. Then $\lambda_1, \, \lambda_2$ have the same sign which, moreover, is opposite to the sign of $\lambda_3$.

\smallbreak

\begin{proof}[Proof of the Claim]
Let us first observe that $\lambda_1, \, \lambda_2, \, \lambda_3$ cannot all have the same sign. Indeed, if this were the case,
then one of the two possibilities below would hold for $\tilf$
\begin{itemize}
\item $\tilf$ is locally linearizable about $p$
\item $\tilf$ possesses a Poincar\'e-Dulac normal form about $p$, \cite{arnold}.
\end{itemize}
In the first case, it can immediately be checked by direct integration that all leaves nearby $p$ accumulate at $p$,
contradicting the complete integrability of the foliation. In the second case, a suitable finite sequence of punctual
blow-ups would yield a singular point of saddle-node type for the corresponding transform of $\fol$. Since the existence
of this type of singularity is not compatible with complete integrability, as seen in the proof of Lemma~\ref{lemmanonvanishing},
it follows that $\lambda_1, \, \lambda_2, \, \lambda_3$ cannot all have the same sign as desired. Thus there exists $i$
such that $\lambda_i$ has opposite sign to the other eigenvalues.

Assume for a contradiction that $\lambda_1, \, \lambda_2$, the eigenvalues associated to the foliation induced on $E$, have
opposite signs. Then $\tilf_E$ admits two separatrices $S_1, \, S_2$. Suppose that $\lambda_1$ and $\lambda_3$ have the same
sign and assume that $S_1$ is the separatrix associated to $\lambda_1$. Since the vector field is completely integrable,
there exists an invariant surface $M$, transverse to $E$ and containing $S_1$, whose associated foliation has eigenvalues
$\lambda_1, \, \lambda_3$ at $p$. In particular its eigenvalues at $p$ have the same sign, implying that all leaves on
$M$ accumulate at $p$. This contradicts Lemma~\ref{lemmanumberseparatrices}. The claim is proved.
\end{proof}

It has just been proved that all singular points $p \in E$ of $\tilf$ are dicritical singular points for $\tilf_E$. In other
words, they are singular points for $\tilf_E$ at which $\tilf_E$ admits infinitely many separatrices. In particular, the
Baum-Bott index of $\tilf_E$ at $p$, which is given by
\[
BB(\tilf_E,p) = \frac{\lambda_1}{\lambda_2} + \frac{\lambda_2}{\lambda_1} + 2 \, ,
\]
is greater than or equal to~$4$. Since $\tilf_E$ has only non-degenerate isolated singular points over $E$, the number of
singular points of $\tilf_E$ is given by $1 + k + k^2$, where $k$ denotes the degree of the foliation $\tilf_E$. Thus
\[
\sum_{i=1}^n BB(\tilf_E, p_i) \geq 4(k^2 + k + 1) \, ,
\]
where $p_1, \, \ldots, \, p_n$ are all the singularities of $\tilf_E$. Nonetheless, the Baum-Bott Theorem says that the sum
of the Baum-Bott indexes for all singular points should be
\[
\sum_{i=1}^n BB(\tilf_E, p) = (k + 2)^2 \, .
\]
The resulting contradiction ends the proof of Theorem~\ref{teodicritical}.
\end{proof}

\vspace{0.2cm}

\begin{flushleft}
{\sc Susana Pinheiro} \\
Centro de Matem\'atica da Universidade do Porto, \\
Portugal\\
spinheiro09@gmail.com \\

\end{flushleft}

\vspace{0.2cm}

\begin{flushleft}
{\sc Helena Reis} \\
Centro de Matem\'atica da Universidade do Porto, \\
Faculdade de Economia da Universidade do Porto, \\
Portugal\\
hreis@fep.up.pt \\

\end{flushleft}


\begin{thebibliography}{Dillo 83}

\bibitem[A]{arnold}
V. Arnold, {\it Chapitres suppl\'ementaires de la th\'eorie des \'equations diff\'erentielles
ordinaires}, ``Mir'', Moscow (1984)

\bibitem[C-K-P]{C-K-P} C. Camacho, N. Kuiper, J. Palis, \emph{The topology of holomorphic flows with singularity},
Inst. Hautes Études Sci. Publ. Math. No. 48 (1978), 5–38.

\bibitem[Ca-Sc]{C-S} L. C\^amara, B. Sc\'ardua, \emph{On the integrability of holomorphic vector fields},
Discrete Contin. Dynam. Systems 25 (2009), 481-493.

\bibitem[C-R-S]{C-R-S} F. Cano, C. Roche; M. Spivakovsky, \emph{Reduction of singularities of three-dimensional
line foliations}, preprint (2010).

\bibitem[C-M]{C-M} D. Cerveau, J.F. Mattei, \emph{Formes Int\'egrables Holomorphes Singuli\`eres}, Ast\'erisque,
97 (1982).

\bibitem[Ch]{chaperon} M. Chaperon, \emph{$C^k$-conjugacy of holomorphic flows near a singularity}, Inst. Hautes
Études Sci. Publ. Math. No. 64 (1986), 143–183.

\bibitem[G]{G} J. Guckenheimer, \emph{Hartman's theorem for complex flows in the Poincar\'e domain},
Compositio Math. 24 (1972), 75–82.

\bibitem[K]{Martine} M. Klughertz, \emph{Existence d'une int\'egrale premi\`ere m\'eromorphe pour des germes de
feuilletages \`a feuilles ferm\'ees du plan complexe}, Topology 31 no. 2 (1992), 255–269

\bibitem[L1]{ladis1} N.N. Ladis, \emph{Topological invariants of complex linear flows}, (Russian)
Differencial'nye Uravnenija 12, no. 12 (1976), 2159–2169, 2299

\bibitem[L2]{ladis2} N.N. Ladis, \emph{Topological equivalence of hyperbolic linear systems}, (Russian)
Differencial'nye Uravnenija 13, no. 2 (1977), 255–264, 379–380

\bibitem[L]{lehmann} D. Lehmann, \emph{R\'esidus des sous-vari\'et\'es invariantes d'un feuilletage singulier},
Ann. Inst. Fourier (Grenoble) 41, no. 1 (1991), 211–258

\bibitem[L-S]{seade} B. Lim\'on, J. Seade, \emph{Morse theory and the topology of holomorphic foliations near an
isolated singularity}, J. Topol. 4, no. 3 (2011), 667–686.

\bibitem[M-M]{M-M} J.F. Mattei, R. Moussu, \emph{Holonomie et int\'egrales premi\`eres}, Ann. Scien. \'Ec. Norm.
Sup. 13, no. 4 (1980), 469-523.

\bibitem[MQ-P]{MQ-P} M. McQuillan, D. Panazzolo, \emph{$(1 + d/dz)^{-1}$}, preprint (2012).

\bibitem[M]{Mo} R. Moussu, \emph{Sur l'existence d'int\'egrales premi\`eres holomorphes},
Ann. Scuola Norm. Sup. Pisa Cl. Sci. (4) 26, no. 4 (1998), 709–717

\bibitem[P]{P} D. Panazzolo, \emph{Resolution of singularities of real-analytic vector fields in dimension three},
Acta Math. 197, no. 2 (2006), 167–289.

\bibitem[R-S]{R-S} J.P. Ramis, Y. Sibuya, \emph{Hukuhara domains and fundamental existence and uniqueness theorems for
asymptotic solutions of Gevrey type}, Asymptotic Analysis. 2  (1989), 39–94.

\bibitem[S1]{Suzuki2} M. Suzuki, \emph{Sur les relations d'\'equivalences ouvertes dans les espaces analytiques},
Ann. \'Ec. Norm. Sup. 4 (1974), 531-542.

\bibitem[S2]{Suzuki1} M. Suzuki, \emph{Sur les int\'egrales premi\`eres de certains feuilletages analytiques complexes},
Lect. Notes in Math. Vol. 670 (1976), 53-58.


\end{thebibliography}
\end{document}